\documentclass[12pt, a4paper]{article}
\usepackage{graphicx} % Required for inserting images
\usepackage{amssymb}
\usepackage{mathrsfs}
\usepackage{amsmath}
\usepackage{amsthm}
\usepackage{appendix}
\usepackage{geometry}
\usepackage{circledsteps}
\usepackage{indentfirst}
\usepackage{xcolor}
\newtheorem{statement}{Statement}[section]

\theoremstyle{definition}
\newtheorem{definition}[statement]{Definition}

\theoremstyle{plain}
\newtheorem{theorem}[statement]{Theorem}
\newtheorem{lemma}[statement]{Lemma}
\newtheorem{proposition}[statement]{Proposition}

\newtheorem{remark}[statement]{Remark}

\numberwithin{equation}{section}

\allowdisplaybreaks[4]
\begin{document}

\title{Ergodicity of reflected stochastic reaction-diffusion equations driven by space-time white noise}

\author{Shijie Shang$^{1}$, Jianliang Zhai$^{1}$,
		Tusheng Zhang$^{1,2}$}
	\footnotetext[1]{\, School of Mathematics, University of Science and Technology of China, Hefei, China. Email: sjshang@ustc.edu.cn (Shijie Shang), zhaijl@ustc.edu.cn (Jianliang Zhai).}
	\footnotetext[2]{\, Department of Mathematics, University of Manchester, Manchester M13 9PL, United Kingdom. Email: tusheng.zhang@manchester.ac.uk }

\date{\today}

\maketitle

\begin{abstract}
\noindent
We consider the reflected stochastic reaction-diffusion equation on $[0,1]$:
\begin{align*}
	\left\{
	\begin{aligned}
		d u(t,x) &=\frac{1}{2}\partial_{xx} u(t,x)dt +b(u(t,x))dt + \sigma(u(t,x)) W(dt,dx)+L(dt,dx),\\
         u(t,x)&\geq 0, \quad t\geq 0, \ x\in [0,1],\\
		u(0,x)&=u_0(x)\geq 0, \quad x\in [0,1], \\
        u(t,0) &= u(t,1) = 0, \quad \forall\  t\geq 0,
	\end{aligned}
	\right.
\end{align*}
where the initial value $u_0$ is non-negative on $[0,1]$ satisfying $u_0(0)=u_0(1)=0$, and $ W(dt,dx)$ is a space-time white noise. The $L$ in the equation is a random measure on $[0,\infty)\times(0,1)$, which is a part of the solution pair $(u, L)$. 
%The coefficients
%$b(\cdot), \sigma(\cdot): \mathbb{R}\rightarrow \mathbb{R}$ are deterministic
%measurable functions. 

In this paper, we establish the existence and uniqueness of invariant measures, 
as well as exponential mixing for the reflected stochastic reaction diffusion equation under the dissipative condition
$$(b(x)-b(y))(x-y)\leq -\alpha (x-y)^2,$$
which include the coefficients having polynomial, even exponential growth. 
The big obstacle of utilizing the dissipative condition is the lack of the It\^{o} formula/energy equality for such equations. To circumvent the problem, we use the newly found method in our paper (arXiv:2606.26619, 2026) to fully exploit comparison principles of reflected stochastic reaction-diffusion equation.
%\\[3mm]
\par\vspace{3mm}

\end{abstract}
\newpage
\tableofcontents
\section{Introduction}
In this paper, we consider the reflected stochastic reaction-diffusion equation on $[0,1]$:
\begin{align}\label{1.1}
	\left\{
	\begin{aligned}
		d u(t,x) &=\frac{1}{2}\partial_{xx} u(t,x)dt +b(u(t,x))dt + \sigma(u(t,x)) W(dt,dx)+L(dt,dx),\quad t\geq 0, \ x\in [0,1],\\
         u(t,x)&\geq 0, \quad t\geq 0, \ x\in [0,1],\\
		u(0,x)&=u_0(x)\geq 0, \quad x\in [0,1], \\
        u(t,0) &= u(t,1) = 0, \quad \forall\  t\geq 0,
	\end{aligned}
	\right.
\end{align}
where the initial value $u_0$ is non-negative on $[0,1]$ and satisfies $u_0(0)=u_0(1)=0$, and $ W(dt,dx)$ is a space-time white noise on a probability space $(\Omega, {\cal F}, \{{\cal F}_t\}_{t\geq 0}, \mathbb{P})$. Here $\{{\cal F}_t\}_{t\geq 0}$ 
is the filtration generated by the white noise $W$ and satisfying the usual conditions. The term $L$ in equation \eqref{1.1} is a random measure on $[0,\infty)\times(0,1)$ and is part of the solution pair $(u,L)$. 
The coefficients
$b(\cdot), \sigma(\cdot): \mathbb{R}\rightarrow \mathbb{R}$ are deterministic
measurable functions. 
The following
definition is taken from \cite{MP,NP92}.
\begin{definition}
A pair $(u, L)$ is said to be a solution of equation \eqref{1.1} if
\begin{itemize}
  \item [(i)] $u$ is a continuous process on $[0,\infty)\times[0,1]$. For any $t\geq 0$ and $x\in [0,1]$, $u(t,x)$ is $\mathcal{F}_t$-measurable, $u(t,x)\geq 0$ and $u(t,0) = u(t,1) = 0$, $\mathbb{P}$-a.s.  
  \item [(ii)] $L$ is a random measure on $[0,\infty)\times (0,1)$ such that
  \begin{itemize}
    \item [(a)] $L(\{t\}\times(0,1))=0$ for any $t\geq 0$.
    \item [(b)] $\int_0^t\int_0^1 x(1-x)L(ds,dx) <\infty$ for any $t\geq 0$.
    \item [(c)] $L$ is $\mathcal{F}_t$ adapted in the sense that for any measurable mapping $\psi$, 
        \[
        \int_0^t\int_0^1 \psi(s,x) L(ds,dx) \text{ is } \mathcal{F}_t \text{-measurable.}
        \]
  \end{itemize}
  \item [(iii)] $(u,L)$ solves \eqref{1.1} in the following sense: for any $t\geq 0$, $\phi\in C^2([0,1])$ with $\phi(0) = \phi(1) = 0$, 
%      \begin{align}\label{260619.2134}
%      \int_0^1 u(t,x)\phi(x)dx = & \int_{0}^{1} u_0(x)\phi(x) dx + \frac{1}{2} \int_{0}^{t}\int_0^1 u(s,x)\phi^{\prime \prime}(x)dxds \nonumber\\ 
%      & + \int_{0}^{t} \int_{0}^{1} b(u(s,x))\phi(x)dxds 
%       + \int_{0}^{t}\int_{0}^{1} \sigma(u(s,x))\phi(x)W(ds,dx) \nonumber\\
%       & + \int_{0}^{t}\int_{0}^{1} \phi(x)L(ds,dx).
%      \end{align}
%      
      \begin{align*}
& (u(t),\phi)-\int_0^t(u(s),\phi'')ds+\int_0^t(b(u(s)),\phi)ds\\
=&(u_0, \phi)+\int_0^t\int_{0}^1\sigma(u(s,x))\phi(x)W(ds,dx)+\int_0^t\int_{0}^1\phi(x)L(ds,dx), \quad \mathbb{P}\text{-a.s.},
\end{align*} 
where $(\cdot,\cdot)$ denotes the scalar product in $L^2([0,1])$, and $u(t):=u(t, \cdot)$.
  \item [(iv)] The support of $L$ is contained in the set $\{(t,x)\in [0,\infty)\times(0,1): u(t,x)=0\}$, that is
  \[
  \int_0^{\infty}\int_{0}^{1} u(t,x) L(dt,dx) = 0,\quad \mathbb{P}\text{-a.s.}
  \]
\end{itemize}

\end{definition}

This type of reflected stochastic partial differential equations (SPDEs) were first studied by Nualart and Pardoux in \cite{NP92} when $\sigma(\cdot)=1$, and by Donati-Martin and Pardoux in \cite{MP} where a minimal solution was obtained for general diffusion coefficients $\sigma$. Then, Xu and Zhang 
in \cite{XZ09} proved the uniqueness of the solutions and also obtained a large deviation principle for the reflected SPDEs. 
Various properties of the solution of equation \eqref{1.1} were studied later in \cite{DMZ,DP,HP,ZA,Z}. SPDEs with reflection can also be used to
model the evolution of random interfaces near a hard wall. It was
proved by T. Funaki and S. Olla in \cite{FO} that the fluctuations
of
a $\nabla \phi$ interface model near a hard wall converge in law to the stationary solution of an SPDE with reflection.\\

The purpose of  this paper is to establish the ergodicity, exponential mixing for reflected  stochastic reaction-diffusion equations (SRDEs). Ergodicity concerns the long-time behaviour of stochastic systems and is one of the most important research topics in the study of stochastic partial differential equations. 
Ergodicity for stochastic evolution equations and SPDEs has been studied extensively; see, for instance, \cite{PG-3,PZ,Z,HM}.
%
%For instance, with the irreducibility in hand, one can obtain the ergodicity by proving the strong Feller property (see \cite{PG-3,PZ,Z}), or the asymptotic strong Feller (see \cite{HM}). 

For reflected stochastic reaction-diffusion equations driven by space-time noise with two reflecting walls, the existence and uniqueness of invariant measures were obtained using a coupling method in \cite{YZ}. Regarding the reflected SPDE (\ref{1.1}), only existence of invariant measures was obtained in \cite{KJ}. So far, there exist no results on the exponential mixing  due to the difficulty of utilizing the dissipativity conditions,  because of the lack of It\^{o} formula/energy equalities.

In this paper, we obtain the exponential contractivity as well as exponential mixing of solutions of reflected SPDEs driven by space-time white noise  under dissipativity assumptions, which could include coefficients with polynomial or exponential growth. 
Because It\^{o} formulae and energy equalities are not available, the existing methods in the literature fail, and it is tricky to make use of the dissipativity condition.  We overcome these issues by using the newly found method in our paper \cite{SZZ} to fully exploit the comparison principles for reflected stochastic partial differential equations.

\vskip 0.5cm

The rest of the paper is organized as follows. In Section 2, we recall the precise framework and introduce the hypotheses. In Section 3, we establish moment estimates for the stochastic convolution which will used later in the paper.
In Section 4, we prove the exponential contractivity of the solutions of reflected SRDEs. In Section 5, we provide results on exponential mixing of reflected SRDEs driven by space-time white noise.

\vskip 0.3cm

\noindent\textbf{Convention on constants.} Throughout the paper, $C$ denotes a generic positive constant whose value may change from line to line. Other constants are denoted by $C_1$, $C_2$, etc. Their precise values are not important. Dependence on parameters will be indicated when needed, for example by $C_T$, $C_p$.
\newpage

\section{Framework}

Throughout this paper, we let $L^p$ denote the standard Lebesgue space $L^p([0,1])$ and denote its norm by $\Vert \cdot\Vert_{L^p}$. Let $C([0,1])$ be the space of continuous functions on $[0,1]$ and set
\begin{align}\label{260605.1736}
  C_0^+([0,1]) := \{h\in C([0,1]): h(x)\geq 0 \text{ for all } x\in [0,1] \text{ and } h(0) = h(1) = 0\}.
\end{align}

It was shown in \cite{NP92, XZ09, WA} that $u$ is a solution $u$ to reflected SPDE (\ref{1.1})  if and only if $u$ satisfies the following integral equation:
\begin{align*}
	u(t,x)=&P_tu_0(x)+\int_{0}^{t}\int_{0}^{1} p_{t-s}(x,y)b(u(s,y))dyds\nonumber\\
	&+ \int_{0}^{t}\int_{0}^{1} p_{t-s}(x,y)\sigma(u(s,y))W(ds,dy) \nonumber\\
& + \int_{0}^{t}\int_{0}^{1} p_{t-s}(x,y)L(ds,dy),\quad  \mathbb{P}\text{-a.s.},
\end{align*}
where $p_t(x,y)$ is the heat kernel of $\frac{1}{2}\partial_{xx}$ on $[0,1]$ with the Dirichlet boundary condition, and
\begin{align}\label{260603.1931}
 P_t f(x) := \int_{0}^{1} p_{t}(x,y) f(y)dy, \quad x\in [0,1].
\end{align}

\vskip 0.6cm

Throughout the paper, we use the following heat kernel estimates:
\begin{gather}
\label{260207.1152}
  \int_{0}^{1} p_t(x,y) dx \leq 1 , \quad \forall\ t> 0, \ y\in [0,1],\\
  \label{260207.2020}
  \int_{0}^{1} p_t(x,y)^2  dx \leq \frac{1}{2\sqrt{\pi t}}, \quad \forall\ t> 0, \ y\in [0,1].
\end{gather}
By H\"{o}lder's inequality and Fubini's theorem, we deduce from (\ref{260207.1152}) that 
for any $p\in [1,\infty]$,
\begin{align}\label{260208.1019}
  \Vert P_t f\Vert_{L^p} \leq \Vert f\Vert_{L^p}.
\end{align}

\vskip 0.4cm
We shall use the following  hypotheses in various places in the paper.
\vskip 0.4cm
\noindent {\bf (H1)} $\sigma$ is Lipschitz, i.e.,  there exists a constant $L_{\sigma}$ such that
\begin{align*}
|\sigma(x) - \sigma(y)| \leq L_{\sigma} |x-y|, \quad\forall\  x,y\in\mathbb{R}.
\end{align*}
\noindent {\bf (H2)} $b$ is dissipative, i.e. for some $\alpha>0$,
\begin{align*}
  (b(x) - b(y))(x-y) \leq -\alpha (x-y)^2, \quad\forall\  x,y\in\mathbb{R}.
\end{align*}

\vskip 0.6cm

\section{Uniform-in-time estimate}

We first establish uniform-in-time estimates for stochastic convolutions, which are of independent interest. 
% These bounds will play a crucial role  in the proof of the main result in next section. 
\begin{proposition}\label{estimates 001}
  Let $\{\sigma(s,y): (s,y)\in\mathbb{R}_+\times [0,1]\}$ be a random field such that the stochastic integral against space-time white noise is well defined. Then for any $p>2$, $\beta>0$ and $\alpha>0$ satisfying $\frac{1}{2p}<\alpha<\frac{1}{4}$, there exists a constant $C_{p,\beta,\alpha}>0$ independent of $t$ such that for any $t\geq 0$,
  \begin{align}\label{101.1}
    & {\mathbb{E}}\left[\sup_{x\in [0,1]}\left|\int_0^t\int_0^1 e^{-\beta (t-s)}p_{t-s}(x,y)\sigma(s,y)\,W(ds,dy) \right|^p\right] \nonumber\\
    \leq & C_{p,\beta,\alpha} \sup_{s\in [0,t]}\left(\int_0^s (s-r)^{-2\alpha-\frac{1}{2}} e^{-2\beta (s-r)} \sup_{z\in[0,1]}\left\Vert \sigma(r,z)\right\Vert_{L^p(\Omega)}^2\,dr\right)^{\frac{p}{2}} .
  \end{align}
In particular, for any $T>0$, 
  \begin{align}\label{260617.2015}
     & \sup_{t\in [0,T]}\left\Vert \sup_{x\in [0,1]}\left|\int_0^t\int_0^1 e^{-\beta (t-s)}p_{t-s}(x,y)\sigma(s,y)\,W(ds,dy) \right|\right\Vert_{L^p(\Omega)} \nonumber\\
    \leq  & C_{p,\beta} \sup_{s\in [0,T]} \sup_{y\in[0,1]}\Vert \sigma(s,y)\Vert_{L^p(\Omega)},
%    \sup_{s\in [0,t]}\left(\int_0^s (s-r)^{-2\alpha-\frac{1}{2}} e^{-2\rho (s-r)} \sup_{z\in[0,1]}\left\Vert \sigma(r,z)\right\Vert_{L^p(\Omega)}^2\,dr\right)^{\frac{p}{2}} .
  \end{align}
where the constant $C_{p,\beta}$ is independent of $T$ and can be bounded by
\begin{align}\label{260618.1614}
  C_{p,\beta} < \frac{\sqrt{4p}}{\pi}\left(\frac{1}{2\sqrt{\pi}}\right)^{\frac{1}{p}+\frac{1}{2}} \frac{1}{\beta^{\frac{1}{4}-\frac{1}{2p}}} \times \left(\Gamma\Big(\frac{p-2}{6p}\Big)\right)^{\frac{3}{2}}, 
\end{align}
where $\Gamma(\cdot)$ is the gamma function.

\end{proposition}

\begin{proof}
It suffices to consider the case where the right hand side of (\ref{101.1}) is finite. 
%Obviously, we can assume that the right hand side of (\ref{101.1}) is finite.
We employ the factorization method (see e.g. \cite{PG-2}). Choose $\alpha$ such that  $\frac{1}{2p}<\alpha<\frac{1}{4}$. This is possible because $p>2$. Let
\begin{align*}
  (J_{\alpha}\sigma)(s,y):&= \int_0^s\int_0^1 (s-r)^{-\alpha} e^{-\beta (s-r)}p_{s-r}(y,z)\sigma(r,z)\,W(\mathrm{d}r,\mathrm{d}z), \\
  (J^{\alpha-1}f)(t,x):&= \frac{\sin\pi\alpha}{\pi}\int_0^t\int_0^1 (t-s)^{\alpha-1} e^{-\beta (t-s)} p_{t-s}(x,y)f(s,y)\,\mathrm{d}s\mathrm{d}y.
\end{align*}
By the stochastic Fubini theorem (see Theorem 2.6 in \cite{WA}),  for any $(t,x)\in\mathbb{R}_+\times[0,1]$,
\begin{align*}
  \int_0^t\int_0^1 e^{-\beta(t-s)}p_{t-s}(x,y)\sigma(s,y)\,W(\mathrm{d}s,\mathrm{d}y)=J^{\alpha-1}(J_{\alpha}\sigma)(t,x).
\end{align*}

\noindent Therefore
\begin{align*}
   \sup_{x\in [0,1]}\left|\int_0^t\int_0^1 e^{-\beta(t-s)} p_{t-s}(x,y)\sigma(s,y)\,W(\mathrm{d}s,\mathrm{d}y)\right| 
  = \sup_{x\in [0,1]}\left|J^{\alpha-1}(J_{\alpha}\sigma)(t,x)\right|, \quad \mathbb{P}\text{-a.s.}.
\end{align*}
By H\"{o}lder's inequality, \eqref{260207.1152} and \eqref{260207.2020}, we have
{\allowdisplaybreaks\begin{align}\label{104.1}
  & {\mathbb{E}}\sup_{x\in [0,1]}\left|\int_0^t\int_0^1 e^{-\beta(t-s)}p_{t-s}(x,y)\sigma(s,y)\,W(\mathrm{d}s,\mathrm{d}y)\right|^p \nonumber\\
  =& {\mathbb{E}}\sup_{x\in [0,1]}\left|\frac{\sin\pi\alpha}{\pi} \int_0^t\int_0^1 (t-s)^{\alpha-1} e^{-\beta(t-s)} p_{t-s}(x,y)J_{\alpha}\sigma(s,y)\,\mathrm{d}s\mathrm{d}y\right|^p \nonumber\\
  \leq & \left|\frac{\sin\pi\alpha}{\pi}\right|^p {\mathbb{E}}\sup_{x\in [0,1]}\bigg\{\int_0^t (t-s)^{\alpha-1} e^{-\beta(t-s)} \nonumber\\
  &~~~~~~~~~~~~~\times\left(\int_0^1 p_{t-s}(x,y)|J_{\alpha}\sigma(s,y)|\,\mathrm{d}y\right)\,\mathrm{d}s\bigg\}^p \nonumber\\
  \leq & \left|\frac{\sin\pi\alpha}{\pi}\right|^p {\mathbb{E}}\sup_{x\in [0,1]}\Bigg\{\int_0^t (t-s)^{\alpha-1} e^{-\beta(t-s)}  \nonumber\\
  &~~~~~~~~~~~~~\times\left(\int_0^1 p_{t-s}(x,y)|J_{\alpha}\sigma(s,y)|^{\frac{p}{2}}\,\mathrm{d}y\right)^{\frac{2}{p}}\,\mathrm{d}s\Bigg\}^p \nonumber\\
  \leq & \left|\frac{\sin\pi\alpha}{\pi}\right|^p {\mathbb{E}}\sup_{x\in [0,1]}\Bigg\{\int_0^t (t-s)^{\alpha-1} e^{-\beta(t-s)} \nonumber\\
  &~~~~~~~~~~~~~\times\left(\int_0^1 p_{t-s}(x,y)^2\,\mathrm{d}y\right)^{\frac{1}{2}\times\frac{2}{p}}\left(\int_0^1|J_{\alpha}\sigma(s,y)|^p\,\mathrm{d}y\right)^{\frac{1}{2}\times\frac{2}{p}}\,\mathrm{d}s\Bigg\}^p \nonumber\\
  \leq & \left|\frac{\sin\pi\alpha}{\pi}\right|^p \frac{1}{2\sqrt{\pi}} {\mathbb{E}}\left\{\int_0^t (t-s)^{\alpha-1-\frac{1}{2p}} e^{-\beta(t-s)} \left(\int_0^1|J_{\alpha}\sigma(s,y)|^p\,\mathrm{d}y\right)^{\frac{1}{p}}\,\mathrm{d}s\right\}^p \nonumber\\
  \leq & \left|\frac{\sin\pi\alpha}{\pi}\right|^p \frac{1}{2\sqrt{\pi}} \left\{ \int_0^t (t-s)^{\alpha-1-\frac{1}{2p}}e^{-\beta(t-s)} \left\Vert \left(\int_0^1|J_{\alpha}\sigma(s,y)|^p\,\mathrm{d}y\right)^{\frac{1}{p}} \right\Vert_{L^p(\Omega)} ds \right\}^p \nonumber\\
  \leq & \left|\frac{\sin\pi\alpha}{\pi}\right|^p \frac{1}{2\sqrt{\pi}} \left\{ \int_0^t (t-s)^{\alpha-1-\frac{1}{2p}}e^{-\beta(t-s)}  \left(\int_0^1 \mathbb{E}|J_{\alpha}\sigma(s,y)|^p\,\mathrm{d}y\right)^{\frac{1}{p}}  ds \right\}^p \nonumber\\
  \leq & \left|\frac{\sin\pi\alpha}{\pi}\right|^p \frac{1}{2\sqrt{\pi}} \left\{ \int_0^t (t-s)^{\alpha-1-\frac{1}{2p}}e^{-\beta(t-s)} ds \right\}^p \sup_{s\in [0,t]} \sup_{y\in [0,1]} \mathbb{E}|J_{\alpha}\sigma(s,y)|^p \nonumber\\
  \leq & C_{p,\beta,\alpha}^{\prime} \sup_{s\in [0,t]} \sup_{y\in [0,1]} \mathbb{E}|J_{\alpha}\sigma(s,y)|^p,
\end{align}}
where we have used the condition $\alpha >\frac{1}{2p}$, so that
\begin{align}\label{260617.1328}
  C_{p,\beta,\alpha}^{\prime}= & \left|\frac{\sin\pi\alpha}{\pi}\right|^p \frac{1}{2\sqrt{\pi}} \times \left(\int_0^{\infty} s^{\alpha-1-\frac{1}{2p}} e^{-\beta s}\,\mathrm{d}s\right)^{p} \nonumber\\
  = & \left|\frac{\sin\pi\alpha}{\pi}\right|^p \frac{1}{2\sqrt{\pi}} \times \left(\frac{1}{\beta^{\alpha - \frac{1}{2p}}} \Gamma\big(\alpha - \frac{1}{2p}\big) \right)^p.
\end{align}

For any fixed $(s,y)\in[0,T]\times[0,1]$, let
\[
Z_t:=\int_0^t\int_0^1 (s-r)^{-\alpha} e^{-\beta(s-r)} p_{s-r}(y,z)\sigma(r,z)\,W(\mathrm{d}r,\mathrm{d}z),\quad t\in[0,s].
\]
Then it is easy to see that $\{Z_t\}_{t\in[0,s]}$ is a real-valued martingale (on the interval $[0, s]$). 
Applying the Burkholder-Davis-Gundy inequality (see Theorem B.1 in \cite{K14}), we have for $t\in[0,s]$,
\begin{align*}
	\mathbb{E}|Z_t|^p \leq & (4p)^{\frac{p}{2}} \mathbb{E}\langle Z\rangle_t^{\frac{p}{2}} \nonumber\\
=&(4p)^{\frac{p}{2}}\mathbb{E}\left(\int_0^t\int_0^1 (s-r)^{-2\alpha} e^{-2\beta(s-r)} p_{s-r}(y,z)^2\sigma(r,z)^2\,\mathrm{d}r\mathrm{d}z\right)^{\frac{p}{2}}.
\end{align*}
\noindent Taking $\frac{2}{p}$-th power on both sides of the above inequality, using (\ref{260207.2020}) and H\"{o}lder's inequality, we get
\begin{align*}
& \Vert J_{\alpha}\sigma(s,y)\Vert_{L^p(\Omega)}^2 
 = \left\Vert Z_s\right\Vert_{L^p(\Omega)}^2 \nonumber\\
 \leq & 4p\left\Vert\int_0^s\int_0^1(s-r)^{-2\alpha} e^{-2\beta (s-r)} p_{s-r}(y,z)^2\sigma(r,z)^2\,\mathrm{d}r\mathrm{d}z\right\Vert_{L^{\frac{p}{2}}(\Omega)}\nonumber\\
 \leq & 4p \int_0^s\int_0^1(s-r)^{-2\alpha} e^{-2\beta(s-r)} p_{s-r}(y,z)^2\left\Vert\sigma(r,z)\right\Vert_{L^{p}(\Omega)}^2\,\mathrm{d}r\mathrm{d}z \nonumber\\
 \leq & 4p\int_0^s(s-r)^{-2\alpha} e^{-2\beta(s-r)}  \left(\int_0^1 p_{s-r}(y,z)^2\,\mathrm{d}z\right)\sup_{z\in[0,1]}\left\Vert\sigma(r,z)\right\Vert_{L^p(\Omega)}^2\,\mathrm{d}r  \nonumber\\
  \leq & 4p\frac{1}{2\sqrt{\pi}} \int_0^s(s-r)^{-2\alpha-\frac{1}{2}} e^{-2\beta(s-r)}  \sup_{z\in[0,1]}\left\Vert\sigma(r,z)\right\Vert_{L^p(\Omega)}^2\,\mathrm{d}r  .
\end{align*}
Therefore we take $\frac{p}{2}$-th power on both sides of the above inequality to obtain
\begin{align}\label{260617.1326}
  & \sup_{s\in [0,t]} \sup_{y\in [0,1]} \mathbb{E}|J_{\alpha}\sigma(s,y)|^p \nonumber\\
\leq & \left(4p\frac{1}{2\sqrt{\pi}}\right)^{\frac{p}{2}} \sup_{s\in [0,t]}\left( \int_0^s(s-r)^{-2\alpha-\frac{1}{2}} e^{-2\beta(s-r)}  \sup_{z\in[0,1]}\left\Vert\sigma(r,z)\right\Vert_{L^p(\Omega)}^2\,\mathrm{d}r \right)^{\frac{p}{2}}.
\end{align}
The condition $\alpha <\frac{1}{4}$ guarantees that the above integral is finite. 
Combining (\ref{104.1}), (\ref{260617.1328}) with (\ref{260617.1326}), we obtain
\begin{align}\label{260621.1310}
  & \mathbb{E}\sup_{x\in [0,1]}\left|\int_0^t\int_0^1 e^{-\beta(t-s)}p_{t-s}(x,y)\sigma(s,y)\,W(\mathrm{d}s,\mathrm{d}y)\right|^p \nonumber\\
  \leq & C_{p,\beta,\alpha} \sup_{s\in [0,t]}\left( \int_0^s(s-r)^{-2\alpha-\frac{1}{2}} e^{-2\beta(s-r)}  \sup_{z\in[0,1]}\left\Vert\sigma(r,z)\right\Vert_{L^p(\Omega)}^2\,\mathrm{d}r \right)^{\frac{p}{2}} ,
\end{align}
where
\begin{align}
  C_{p,\beta,\alpha} = & C_{p,\beta,\alpha}^{\prime}\times \left(4p\frac{1}{2\sqrt{\pi}}\right)^{\frac{p}{2}} \nonumber\\
  = & \left|\frac{\sin\pi\alpha}{\pi}\right|^p (4p)^{\frac{p}{2}} \left(\frac{1}{2\sqrt{\pi}}\right)^{\frac{p}{2}+1} \times \left(\frac{1}{\beta^{\alpha - \frac{1}{2p}}} \Gamma\big(\alpha - \frac{1}{2p}\big)\right)^p.
\end{align}
This proves (\ref{101.1}).

Taking $\frac{1}{p}$-th power on \eqref{260621.1310} and then taking the supremum over $t\in [0,T]$, we obtain
\begin{align*}
  & \sup_{t\in [0,T]}\left(\mathbb{E}\sup_{x\in [0,1]}\left|\int_0^t\int_0^1 e^{-\beta(t-s)}p_{t-s}(x,y)\sigma(s,y)\,W(\mathrm{d}s,\mathrm{d}y)\right|^p\right)^{\frac{1}{p}} \nonumber\\
  \leq & C_{p,\beta,\alpha}^{\frac{1}{p}} \sup_{s\in [0,T]}\left( \int_0^s(s-r)^{-2\alpha-\frac{1}{2}} e^{-2\beta(s-r)}  \sup_{z\in[0,1]}\left\Vert\sigma(r,z)\right\Vert_{L^p(\Omega)}^2\,\mathrm{d}r \right)^{\frac{1}{2}} \nonumber\\
  \leq & C_{p,\beta,\alpha}^{\frac{1}{p}}\sup_{s\in [0,T]}\left( \int_0^s(s-r)^{-2\alpha-\frac{1}{2}} e^{-2\beta(s-r)}  \,\mathrm{d}r \right)^{\frac{1}{2}} \times \sup_{r\in [0,T]}\sup_{z\in[0,1]}\left\Vert\sigma(r,z)\right\Vert_{L^p(\Omega)} \nonumber\\
  \leq & C_{p,\beta} \sup_{r\in [0,T]}\sup_{z\in[0,1]}\left\Vert\sigma(r,z)\right\Vert_{L^p(\Omega)} ,
\end{align*}
where 
\begin{align*}
  C_{p,\beta} = & \min_{\frac{1}{2p}<\alpha<\frac{1}{4}} \left[ C_{p,\beta,\alpha}^{\frac{1}{p}}\sup_{s\in [0,T]}\left( \int_0^s r^{-2\alpha-\frac{1}{2}} e^{-2\beta r} dr \right)^{\frac{1}{2}} \right] \nonumber\\
  \leq & \min_{\frac{1}{2p}<\alpha<\frac{1}{4}} \left[\left|\frac{\sin\pi\alpha}{\pi}\right| \sqrt{4p}  \left(\frac{1}{2\sqrt{\pi}}\right)^{\frac{1}{2}+\frac{1}{p}} \times \frac{1}{\beta^{\alpha - \frac{1}{2p}}} \Gamma\Big(\alpha - \frac{1}{2p}\Big) \times \left( \frac{1}{\beta^{\frac{1}{2} - 2\alpha}} \Gamma\Big(\frac{1}{2} - 2\alpha\Big) \right)^{\frac{1}{2}} \right] \nonumber\\
  \leq & \frac{\sqrt{4p}}{\pi}\left(\frac{1}{2\sqrt{\pi}}\right)^{\frac{1}{2}+\frac{1}{p}} \frac{1}{\beta^{\frac{1}{4}-\frac{1}{2p}}} \times \left(\Gamma\Big(\frac{p-2}{6p}\Big)\right)^{\frac{3}{2}}. 
\end{align*}
Here we have used
\begin{align}\label{260621.1425}
  \min_{\frac{1}{2p}<\alpha<\frac{1}{4}} \left[\Gamma\Big(\alpha -\frac{1}{2p}\Big) \times \Big(\Gamma\Big(\frac{1}{2} -2\alpha\Big)\Big)^{\frac{1}{2}}\right] = \left(\Gamma\Big(\frac{p-2}{6p}\Big)\right)^{\frac{3}{2}}.
\end{align}
This is because
\[
\frac{d}{d\alpha}\log \left[\Gamma\Big(\alpha -\frac{1}{2p}\Big) \times \Big(\Gamma\Big(\frac{1}{2} -2\alpha\Big)\Big)^{\frac{1}{2}}\right] = \psi\big(\alpha -\frac{1}{2p}\big) - \psi\big(\frac{1}{2}-2\alpha\big), 
\]
where $\psi(x) = \frac{\Gamma^{\prime}(x)}{\Gamma(x)}$ is the digamma function, which is strictly increasing on $(0,\infty)$. 
Since the trigamma function $\psi^{\prime}$ is strictly positive, the minimum in \eqref{260621.1425} is achieved when $\alpha -\frac{1}{2p} = \frac{1}{2} -2\alpha$, which is equivalent to $\alpha = \frac{p+1}{6p}$. This completes the proof of Proposition \ref{estimates 001}.

\end{proof}

\section{Exponential contractivity}

In this section, we assume that \eqref{1.1} admits a unique solution. We show that the solutions of \eqref{1.1} are exponentially contracting when the drift is  dissipative. We stress that the It\^{o} formula/energy equality is not available for 
such stochastic reaction-diffusion equations driven by space-time white noise. It is tricky to make use of the dissipativity condition (H2).
%Assume equation (\ref{3.1}) has a unique solution in $L^1_{\lambda}$,
\begin{theorem}\label{260206.2029}
Let $u^i(t,\cdot)$ denote the solution to equation \eqref{1.1} with initial value $u_0^i$, $i=1,2$. 
\begin{itemize}
  \item [(i)] Suppose that (H2) holds. 
%  Assume that $\mathbb{E}\int_{\mathbb{R}}|u(t,x)|e^{-\lambda|x|} dx <\infty$ for any $t\geq 0$. 
Then for any $ t\geq 0$,
\begin{align}\label{260207.1931}
    \mathbb{E} \Vert u^2(t) - u^1(t)\Vert_{L^1} \leq 2 e^{-\alpha t} \mathbb{E}\Vert u_0^2 - u_0^1 \Vert_{L^1} .
\end{align}
  \item [(ii)] Suppose that (H1) and (H2) hold. Assume that
   \begin{align}\label{260629.1108}
     \sup_{t\in [0,T]}\mathbb{E}\int_{0}^{1}|u^2(t,x) - u^1(t,x)|^2  dx <\infty, \qquad \forall \ T\geq 0.
   \end{align}

  If $\alpha\geq 0$ when $L_{\sigma} = 0$ or $\alpha>\frac{L_{\sigma}^4}{8}$ when $L_{\sigma} > 0$, 
then for any $\kappa$ satisfying 
\begin{align}\label{260401.1108}
  \begin{cases}
     0\leq \kappa \leq \alpha, & \text{ when } L_{\sigma} = 0,\\
    0\leq \kappa<\alpha - \frac{L_{\sigma}^4}{8}, & \text{ when } L_{\sigma} > 0,
  \end{cases}
\end{align}
there exists a constant $C_{\alpha,\kappa}>0$ such that for any $t\geq 0$,
\begin{align}\label{260208.1116}
\left(\mathbb{E} \int_{0}^{1} |u^2(t,x) - u^1(t,x)|^2  dx\right)^{\frac{1}{2}} \leq 2 C_{\alpha,\kappa} e^{-\kappa t}\left(\mathbb{E}\Vert  u^2_0 - u^1_0 \Vert_{L^2}^2\right)^{\frac{1}{2}}.
\end{align}
    \item [(iii)] Suppose that (H1) and (H2) hold. Assume that for some $p>2$, 
    \begin{align}\label{260620.2115}
      \sup_{t\in [0,T]}\sup_{x\in [0,1]}\mathbb{E}\left[|u^2(t,x) - u^1(t,x)|^p\right] <\infty \quad\text{for any }\ T\geq 0.
    \end{align}
  If 
  \begin{align*}
    \alpha > \left\{\frac{\sqrt{4p}}{\pi}\left(\frac{1}{2\sqrt{\pi}}\right)^{\frac{1}{p}+\frac{1}{2}} \left(\Gamma\Big(\frac{p-2}{6p}\Big)\right)^{\frac{3}{2}} \times L_{\sigma}\right\}^{\frac{4p}{p-2}},
  \end{align*}
  then for any $\kappa$ satisfying 
  \begin{align}\label{260621.1558}
    0<\kappa< \alpha - \left\{\frac{\sqrt{4p}}{\pi}\left(\frac{1}{2\sqrt{\pi}}\right)^{\frac{1}{p}+\frac{1}{2}} \left(\Gamma\Big(\frac{p-2}{6p}\Big)\right)^{\frac{3}{2}} \times L_{\sigma}\right\}^{\frac{4p}{p-2}},
  \end{align}
there exists a constant $C_{p,\alpha,\kappa}>0$ such that for any $t\geq 0$,
\begin{align}\label{260620.2004}
\left(\mathbb{E} \Big[\sup_{x\in [0,1]} |u^2(t,x) - u^1(t,x)|^p \Big]  \right)^{\frac{1}{p}} \leq 2 C_{p,\alpha,\kappa} e^{-\kappa t}\left(\mathbb{E}\Big[ \sup_{x\in [0,1]} | u^2_0 - u^1_0 |^p \Big]\right)^{\frac{1}{p}}.
\end{align}
\end{itemize}
\end{theorem}

\begin{proof}

The proof is divided into two steps.

\textbf{Step 1.} We prove Theorem \ref{260206.2029} under the restriction that $u_0^2(x)\geq u^1_0(x)$ for all $x\in [0,1]$. 

In this case, the comparison theorem gives $u^2(t,x)\geq u^1(t,x)$ for any $(t,x)\in\mathbb{R}_+\times [0,1]$. 
The semigroup generated by $\frac{1}{2}\partial_{xx} - \alpha I$ is $e^{-\alpha t}P_t$, where $P_t$ is defined in (\ref{260603.1931}). 
Hence, the solution to (\ref{1.1}) admits the following mild form:
%So we write the solution to equation (\ref{3.1}) in a mild form:
\begin{align*}
  u^i(t,x) = & e^{-\alpha t} P_t u_0^i(x) + \int_{0}^{t} \int_{0}^{1} e^{-\alpha (t-s)} p_{t-s}(x,y)b_{\alpha}(u^i(s,y)) dsdy \nonumber\\
  & + \int_{0}^{t} \int_{0}^{1} e^{-\alpha (t-s)} p_{t-s}(x,y)\sigma (u^i(s,y)) W(ds,dy)  \nonumber\\
  & + \int_{0}^{t} \int_{0}^{1} e^{-\alpha (t-s)} p_{t-s}(x,y)L^i(ds,dy),
\end{align*}
where $b_{\alpha}(u)=b(u)+\alpha u$. The crucial observation is that (H2) holds if and only if $b_{\alpha}(\cdot)$ is decreasing. 
By subtraction, we have
\begin{align*}
  u^2(t,x) -u^1(t,x) = & e^{-\alpha t} P_t (u_0^2-u^1_0)(x) \nonumber\\
  & + \int_{0}^{t} \int_{0}^{1} e^{-\alpha (t-s)} p_{t-s}(x,y)[b_{\alpha}(u^2(s,y)) - b_{\alpha}(u^1(s,y))] dsdy \nonumber\\
  & + \int_{0}^{t} \int_{0}^{1} e^{-\alpha (t-s)} p_{t-s}(x,y)[\sigma (u^2(s,y)) - \sigma (u^1(s,y))]W(ds,dy) \\
  & + \int_{0}^{t} \int_{0}^{1} e^{-\alpha (t-s)} p_{t-s}(x,y)[L^2(ds,dy) - L^1(ds,dy)].
\end{align*}
By the proof of Theorem 1.4 of \cite{NP92}, $L^2(ds,dy)-L^1(ds,dy)$ is a negative measure. This yields
\begin{align}
  \int_{0}^{t} \int_{0}^{1} e^{-\alpha (t-s)} p_{t-s}(x,y)[L^2(ds,dy) - L^1(ds,dy)] \leq 0.
\end{align}
Since $u^2(s,y)\geq u^1(s,y)$ and $b_{\alpha}$ is decreasing, 
we deduce that
%By (\ref{260206.2004}), $b_{\alpha}$ is a decreasing function, it follows from $u^2(s,y)\geq u^1(s,y)$ and (\ref{260207.1135}) that
\begin{align}\label{260207.2002}
  u^2(t,x) - u^1(t,x) \leq & e^{-\alpha t} P_t (u_0^2-u^1_0)(x) \nonumber\\
  & + \int_{0}^{t} \int_{0}^{1} e^{-\alpha (t-s)} p_{t-s}(x,y)[\sigma (u^2(s,y)) - \sigma (u^1(s,y))]W(ds,dy).
\end{align}
Taking expectations in the above inequality, we obtain
\begin{align*}
  \mathbb{E}|u^2(t,x) - u^1(t,x)| = \mathbb{E}[u^2(t,x) - u^1(t,x)] \leq e^{-\alpha t}\mathbb{E}[ P_t (u_0^2-u^1_0)(x)].
\end{align*}
Applying (\ref{260208.1019}) with $p=1$ and Fubini's theorem yield
\begin{align*}
   \mathbb{E}\Vert u^2(t) - u^1(t) \Vert_{L^1} \leq e^{-\alpha  t} \mathbb{E}\Vert u_0^2 - u_0^1 \Vert_{L^1} .
\end{align*}
%Therefore, for any $\alpha>0$, we can always take sufficiently small $\lambda>0$ such that $\lambda<$
This proves (\ref{260207.1931}) in the case $u_0^2\geq u^1_0$.

\vskip 0.3cm

Then, we prove (\ref{260208.1116}) under the condition $u_0^2\geq u^1_0$. 
Taking the $L^2(\Omega)$-norm on both sides of (\ref{260207.2002}), we obtain
\begin{align*}
  & \Vert u^2(t,x) -u^1(t,x) \Vert_{L^2(\Omega)} \nonumber\\
   \leq & e^{-\alpha t} \Vert P_t(u^2_0 - u^1_0)(x) \Vert_{L^2(\Omega)} + \left\Vert \int_{0}^{t}\int_{0}^{1} e^{-\alpha (t-s)} p_{t-s}(x,y)[\sigma (u^2(s,y)) - \sigma (u^1(s,y))]W(ds,dy) \right\Vert_{L^2(\Omega)} \\
%  = & e^{-\alpha t} \Vert P_t(u^2_0 - u^1_0)(x) \Vert_{L^2(\Omega)} + \left\{\mathbb{E} \int_{0}^{t}\int_{\mathbb{R}} e^{-2\alpha (t-s)} p_{t-s}(x,y)^2[\sigma (u^2(s,y)) - \sigma (u^1(s,y))]^2 dsdy \right\}^{\frac{1}{2}} \\
  \leq & e^{-\alpha t} \Vert P_t(u^2_0 - u^1_0)(x) \Vert_{L^2(\Omega)} + \left\{ \mathbb{E}\int_{0}^{t}\int_{0}^{1} e^{-2\alpha (t-s)} p_{t-s}(x,y)^2 L_{\sigma}^2 | u^2(s,y) -  u^1(s,y)|^2 dsdy \right\}^{\frac{1}{2}} .
\end{align*}
We then take the $L^2$-norm with respect to the spatial variable $x$ on both sides of the above inequality to get
\begin{equation*}
  \begin{aligned}
  & \left(\mathbb{E}\Vert u^2(t) - u^1(t)\Vert_{L^2}^2\right)^{\frac{1}{2}} \leq  e^{-\alpha t} \left(\mathbb{E}\Vert P_t(u^2_0 - u^1_0)\Vert_{L^2}^2\right)^{\frac{1}{2}} \\
 & + \left\{\mathbb{E}\int_{0}^{t}\int_{0}^{1} \int_{0}^{1} e^{-2\alpha (t-s)} p_{t-s}(x,y)^2 L_{\sigma}^2 \vert u^2(s,y) -  u^1(s,y)\vert^2  dsdydx \right\}^{\frac{1}{2}}.
\end{aligned}
\end{equation*}
By (\ref{260208.1019}), Fubini's theorem and (\ref{260207.2020}), we have
\begin{align}\label{260207.2034}
  & \left(\mathbb{E}\Vert u^2(t) - u^1(t)\Vert_{L^2}^2\right)^{\frac{1}{2}} \leq  e^{-\alpha t} \left(\mathbb{E}\Vert u^2_0 - u^1_0\Vert_{L^2}^2\right)^{\frac{1}{2}} \nonumber\\
 & + \left\{ \int_{0}^{t}  \frac{L_{\sigma}^2}{2\sqrt{\pi (t-s)}}  e^{-2\alpha (t-s)} \mathbb{E}\int_{0}^{1} \vert u^2(s,y) -  u^1(s,y)\vert^2 dyds \right\}^{\frac{1}{2}}.
\end{align}
For $\kappa\geq 0$, set
\[
\mathcal{N}_T^{\kappa}(u):= \sup_{0\leq t\leq T} \left[ e^{\kappa t} \left(\mathbb{E} \int_{0}^{1} |u(t,x)|^2 dx\right)^{\frac{1}{2}} \right].
\]
It follows from (\ref{260207.2034}) that
\begin{align}\label{260208.1055}
  & \mathcal{N}_T^{\kappa}(u^2 -u^1) \nonumber\\
  \leq & \sup_{0\leq t\leq T} \left[e^{-(\alpha -\kappa) t}\right] \left(\mathbb{E}\Vert  u^2_0 - u^1_0 \Vert_{L^2}^2\right)^{\frac{1}{2}} \nonumber\\
  & + \sup_{0\leq t\leq T}\left\{\int_{0}^{t}  \frac{L_{\sigma}^2}{2\sqrt{\pi (t-s)}}  e^{-2\alpha (t-s)} e^{2\kappa (t-s)} \mathcal{N}_T^{\kappa}(u^2 -u^1)^2 ds \right\}^{\frac{1}{2}} \nonumber\\
 \leq & \sup_{0\leq t\leq T} \left[e^{-(\alpha -\kappa) t}\right] \left(\mathbb{E}\Vert  u^2_0 - u^1_0 \Vert_{L^2}^2\right)^{\frac{1}{2}} \nonumber\\
  & + \mathcal{N}_T^{\kappa}(u^2 -u^1) \times \sup_{t\geq 0}\left\{ \frac{L_{\sigma}^2}{2\sqrt{\pi}}\int_{0}^{t}  \frac{1}{\sqrt{ (t-s)}}e^{- 2(\alpha-\kappa)(t-s)}  ds \right\}^{\frac{1}{2}} \nonumber\\
  \leq & \sup_{0\leq t\leq T} \left[e^{-(\alpha -\kappa) t}\right] \left(\mathbb{E}\Vert  u^2_0 - u^1_0 \Vert_{L^2}^2\right)^{\frac{1}{2}}  + \mathcal{N}_T^{\kappa}(u^2 -u^1) \times \left[8(\alpha-\kappa)\right]^{-\frac{1}{4}} L_{\sigma}.
\end{align}
By condition (\ref{260401.1108}), we have
\begin{align*}
\alpha   -\kappa \geq 0 \quad\text{ and }\quad \left[8(\alpha-\kappa)\right]^{-\frac{1}{4}} L_{\sigma} <1 .
\end{align*}
Hence (\ref{260208.1055}) implies that
\begin{align}\label{260605.1351}
  \mathcal{N}_T^{\kappa}(u^2 -u^1) \leq \left(\mathbb{E}\Vert  u^2_0 - u^1_0 \Vert_{L^2}^2\right)^{1/2} +  \mathcal{N}_T^{\kappa}(u^2 -u^1) \left[8(\alpha-\kappa)\right]^{-\frac{1}{4}} L_{\sigma}.
\end{align}
%Squaring both sides of (\ref{260207.2034}), using the fact that $\mathbb{E}\Vert u^2(t) - u^1(t)\Vert_{L^2}^2<\infty$ for any $t\geq 0$, and applying Gronwall's inequality, we find that 
\eqref{260629.1108} means that 
$\mathcal{N}^{\kappa}_T(u^2 -u^1)<\infty$ for any $T>0$. Hence (\ref{260605.1351}) gives
\begin{align*}
  \sup_{T\geq 0}\mathcal{N}_T^{\kappa}(u^2 -u^1) \leq  C_{\alpha,\kappa} \left(\mathbb{E}\Vert  u^2_0 - u^1_0 \Vert_{L^2}^2\right)^{\frac{1}{2}},
\end{align*}
where 
%the constant $C_{\alpha,\lambda,\kappa}$ is
\begin{align*}
    C_{\alpha,\kappa}:= \left(1-\left[8\big(\alpha-\kappa\big)\right]^{-\frac{1}{4}} L_{\sigma} \right)^{-1}.
\end{align*}
In particular, we obtain
\begin{align*}
\left(\mathbb{E} \int_{0}^{1} |u^2(t,x) - u^1(t,x)|^2  dx\right)^{\frac{1}{2}} \leq C_{\alpha,\kappa} e^{-\kappa t}\left(\mathbb{E}\Vert  u^2_0 - u^1_0 \Vert_{L^2}^2\right)^{1/2}.
\end{align*}
This proves (\ref{260208.1116}) in the case of $u_0^2\geq u^1_0$.

\vskip 0.3cm

Next, we prove (\ref{260620.2004}) under the condition $u_0^2\geq u^1_0$. 
For $\kappa\geq 0$, set
\[
\mathcal{N}_{p,T}^{\kappa}(u):= \sup_{0\leq t\leq T} \left[ e^{\kappa t} \left(\mathbb{E} \Big[ \sup_{x\in [0,1]} |u(t,x)|^p \Big]\right)^{\frac{1}{p}} \right].
\]
Taking $\sup_{x\in [0,1]}$ on both sides of (\ref{260207.2002}) gives
\begin{align*}
  & \sup_{x\in [0,1]}| u^2(t,x) - u^1(t,x)| =  \sup_{x\in [0,1]}[ u^2(t,x) - u^1(t,x)]  \nonumber\\
  \leq & e^{-\alpha t} \sup_{x\in [0,1]} |P_t (u_0^2-u^1_0)(x)| \nonumber\\
  & + \sup_{x\in [0,1]} \left|\int_{0}^{t} \int_{0}^{1} e^{-\alpha (t-s)} p_{t-s}(x,y)[\sigma (u^2(s,y)) - \sigma (u^1(s,y))]W(ds,dy)\right|.
\end{align*}
%Then taking the $L^p(\Omega)$-norm on both sides of the above inequality, by \eqref{260208.1019} we have
Therefore,
\begin{align*}
  & \mathcal{N}_{p,T}^{\kappa}(u^2 - u^1) \leq \sup_{t\in [0,T]} e^{-(\alpha-\kappa)t} \cdot \sup_{x\in [0,1]}| u_0^2(x) - u_0^1(x)| \nonumber\\
  & + \sup_{t\in [0,T]}\left\Vert \sup_{x\in [0,1]} \left|\int_{0}^{t} \int_{0}^{1} e^{-(\alpha-\kappa) (t-s)} p_{t-s}(x,y)\big[e^{\kappa s} \big(\sigma (u^2(s,y)) - \sigma (u^1(s,y))\big)\big]W(ds,dy)\right| \right\Vert_{L^p(\Omega)}, 
\end{align*}
where we have used \eqref{260208.1019}. 
The stochastic convolution can be estimated by \eqref{260617.2015}. Thus
\begin{align}\label{260620.2245}
 \mathcal{N}_{p,T}^{\kappa}(u^2 - u^1) 
  \leq & \sup_{t\in [0,T]} e^{-(\alpha-\kappa)t} \cdot \sup_{x\in [0,1]}| u_0^2(x) - u_0^1(x)| \nonumber\\
 & + C_{p,\alpha-\kappa} \sup_{s\in [0,T]}\sup_{y\in [0,1]} \big\Vert e^{\kappa s} \big( \sigma(u^2(s,y))  - \sigma(u^1(s,y)) \big) \big\Vert_{L^p(\Omega)} \nonumber\\
  \leq & \sup_{t\in [0,T]} e^{-(\alpha-\kappa)t} \cdot \sup_{x\in [0,1]}| u_0^2(x) - u_0^1(x)| \nonumber\\
  & + C_{p,\alpha-\kappa}L_{\sigma}\sup_{s\in [0,T]}\left(e^{\kappa s}\sup_{y\in [0,1]} \big\Vert  u^2(s,y)  -  u^1(s,y) \big\Vert_{L^p(\Omega)}\right) ,
\end{align}
where the constant $C_{p,\alpha-\kappa}$ is the constant appearing in \eqref{260617.2015} with $\beta$ replaced by $\alpha-\kappa$. 
It follows from \eqref{260620.2115} that for any $0< \kappa<\alpha$, 
\begin{align*}
  \mathcal{N}_{p,T}^{\kappa}(u^2 - u^1) < \infty, \quad \forall \ T\geq 0.
\end{align*}
By \eqref{260620.2245}, 
\begin{align*}
 \mathcal{N}_{p,T}^{\kappa}(u^2 - u^1) \leq  \sup_{t\in [0,T]} e^{-(\alpha-\kappa)t} \cdot \sup_{x\in [0,1]}| u_0^2(x) - u_0^1(x)| + C_{p,\alpha-\kappa} L_{\sigma} \mathcal{N}_{p,T}^{\kappa}(u^2 - u^1) .
\end{align*}
Note that \eqref{260621.1558} is equivalent to $0<\kappa <\alpha$ and $C_{p,\alpha-\kappa} L_{\sigma} <1$. 
Hence 
\begin{align*}
  \sup_{T\geq 0}\mathcal{N}_{p,T}^{\kappa}(u^2 - u^1) \leq (1 - C_{p,\alpha-\kappa}L_{\sigma})^{-1} \sup_{x\in [0,1]}| u_0^2(x) - u_0^1(x)|,
\end{align*}
which, in particular, implies \eqref{260620.2004}.

\vskip 0.6cm

\textbf{Step 2.} We remove the restriction $u_0^2(x)\geq u_0^1(x)$, $x\in [0,1]$.

Let $u^{(1,2)}(t,x)$ denote the solution to equation (\ref{1.1}) with initial value $u_0^1(x)\vee u_0^2(x)$. Note that the constants appearing in (\ref{260207.1931}) and (\ref{260208.1116}) are independent of the initial values. Hence, for $i=1,2$,
\begin{align*}
  \left(\mathbb{E} \int_{0}^{1} |u^i(t,x) - u^{(1,2)}(t,x)|^2  dx\right)^{\frac{1}{2}} \leq C_{\alpha,\kappa} e^{-\kappa t}\left(\mathbb{E}\Vert  u^i_0 - u^1_0\vee u^2_0 \Vert_{L^2}^2\right)^{\frac{1}{2}}.
\end{align*}
Note that for $i=1,2$,
\begin{align*}
  |u_0^i(x) -(u_0^1\vee u_0^2)(x)| \leq |u_0^2(x) - u_0^1(x)|.
\end{align*}
Hence, by the triangle inequality we obtain
\begin{align*}
    \left(\mathbb{E} \int_{0}^{1} |u^2(t,x) - u^{1}(t,x)|^2  dx\right)^{\frac{1}{2}} \leq 2 C_{\alpha,\kappa} e^{-\kappa t}\left(\mathbb{E}\Vert  u^2_0 - u^1_0 \Vert_{L^2}^2\right)^{\frac{1}{2}}.
\end{align*}
This completes the proof of \eqref{260208.1116}. The proofs of  \eqref{260207.1931} and \eqref{260620.2004} in this case are  similar, and we omit the details.
\end{proof}

\section{Exponential mixing}
In this section, we first prove an abstract exponential mixing theorem for reflected stochastic reaction-diffusion equations. We then give concrete conditions under which the hypotheses of the abstract theorem are fulfilled. 
Throughout this section, unless otherwise stated, we assume by default that \eqref{1.1} admits a unique solution. 
\begin{theorem}\label{260209.1719}
Let $u$ be the solution of equation \eqref{1.1} with initial value $f$. 
\begin{itemize}
  \item [(i)] Suppose that (H2) holds. If $\alpha > 0$ and there exists $f\in L^1$ such that
\begin{align}\label{260209.1121-1}
  \sup_{t\geq 0}\mathbb{E} \Vert u(t)\Vert_{L^1} <\infty ,
\end{align}
then there exists a unique invariant measure in $L^1$. Moreover, the solution is exponentially mixing with respect to the Wasserstein metric $W_1$.
  \item [(ii)] Suppose that (H1) and (H2) hold. 
  If $\alpha>\frac{L_{\sigma}^4}{8}$ and there exists $f\in L^2$ such that
\begin{align}\label{260209.1121-2}
  \sup_{t\geq 0}\mathbb{E} [\Vert u(t)\Vert^2_{L^2}] <\infty ,
\end{align}
then there exists a unique invariant measure in $L^2$. Moreover, the solution is exponentially mixing with respect to the Wasserstein metric  $W_2$.
  \item [(iii)] Suppose that (H1) and (H2) hold. 
    If there exists $p>2$ such that
  \begin{align*}
    \alpha > \left\{\frac{\sqrt{4p}}{\pi}\left(\frac{1}{2\sqrt{\pi}}\right)^{\frac{1}{p}+\frac{1}{2}} \left(\Gamma\Big(\frac{p-2}{6p}\Big)\right)^{\frac{3}{2}} \times L_{\sigma}\right\}^{\frac{4p}{p-2}},
  \end{align*}
  and there exists $f\in C_0^+([0,1])$ such that
\begin{align}\label{260621.1650}
  \sup_{t\geq 0}\mathbb{E} \Big[\sup_{x\in [0,1]} | u(t,x)|^p\Big] <\infty ,
\end{align}
then there exists a unique invariant measure in $C_0^+([0,1])$. Moreover, the solution is exponentially mixing with respect to the Wasserstein metric $W_p$.
\end{itemize}
\end{theorem}

\begin{proof} We prove only (ii), since the proofs of (i) and (iii) are similar. 
Let $W(t,x)$ be the Brownian sheet corresponding to the space-time white noise $W(dt,dx)$, and let $W_1(t,x)$ be another Brownian sheet on $[0,\infty)\times [0,1]$ that is independent of $W(t,x)$. Set
\begin{align*}
  \overline{W}(t,x) :=
  \begin{cases}
    W(t,x), & \mbox{if} \quad  t\geq 0,\  x\in [0,1], \\
    W_1(-t,x), & \mbox{if} \quad t<0, \ x\in [0,1].
  \end{cases}
\end{align*}
Let $\overline{\mathcal{F}}_t$ be the filtration generated by $\{\overline{W}(s,x): s\leq t, \ x\in [0,1]\}$ and satisfying the usual conditions. For any $\delta \geq 0$, consider the following SPDE:
\begin{align}\label{260209.1109}
\begin{cases}
    \partial_t u(t,x) = \frac{1}{2}\partial_{xx} u(t,x) + b(u(t,x)) +  \sigma(u(t,x)) \overline{W}(dt,dx) +  \overline{L}(dt,dx), \quad x\in [0,1],\\
    u(t,x)\geq 0, \quad t\geq -\delta, \ x\in [0,1],\\
    u(-\delta,x) = f(x), \quad x\in [0,1], \\
    u(t,0) = u(t,1) = 0, \quad \forall \ t\geq -\delta.
\end{cases}
\end{align}
Then (\ref{260209.1109}) admits a unique solution, 
%Due to the Lipschitz continuity of $b$ and $\sigma$, there exists a unique solution to equation (\ref{260209.1109}), 
which we denote by $u_{-\delta}(t,x)$ for $t\geq -\delta$.
By condition (\ref{260209.1121-2}) and the fact that 
the law $\mathcal{L}(u_{-\delta}(t))$ of $u_{-\delta}(t)$ equals the law of $u(t+\delta)$ on $L^2$, 
there exists a constant $C>0$ independent of $t$ and $\delta$ such that
\begin{align}\label{260209.1547}
  \mathbb{E} [ \Vert u_{-\delta}(t)\Vert_{L^2}^2 ] \leq C, \quad \forall\  \delta\geq 0 , \  t\geq -\delta .
\end{align}
For $\delta>\gamma$, consider the solutions $u_{-\delta}$ and $u_{-\gamma}$ on the interval $[-\gamma,\infty)$. Arguing as in the proof of Theorem \ref{260206.2029}, we obtain
\begin{equation}\label{260209.1618}
\begin{aligned}
  \left(\mathbb{E} [\Vert u_{-\delta}(t) - u_{-\gamma}(t)\Vert_{L^2}^2 ]\right)^{\frac{1}{2}} \leq & 2 C_{\alpha,\kappa} e^{-\kappa(t+\gamma)} \left(\mathbb{E}[\Vert  u_{-\delta}(-\gamma) - f\Vert_{L^2}^2] \right)^{\frac{1}{2}} \\
  \leq &  2 C_{\alpha,\kappa} e^{-\kappa(t+\gamma)} \left(\Vert f\Vert_{L^2} + C \right), \quad t\geq -\gamma, 
\end{aligned}
\end{equation}
where we have used (\ref{260209.1547}) in the last line. This implies that for any $t\in\mathbb{R}$, the sequence of random variables $\{u_{-\gamma}(t)\}_{\gamma\geq 0}$ is a Cauchy sequence in $L^2(\Omega, L^2)$ as $\gamma\rightarrow\infty$. Let
\begin{align*}
  \xi := \lim_{\gamma\rightarrow\infty} u_{-\gamma}(0), \quad \mu := \mathcal{L}(\xi) \text{ on }  L^2.
\end{align*}
Then $\mu$ is the unique invariant measure and is independent of the initial value $f$. In fact, the law
\begin{align*}
 \mu_t := \mathcal{L} (u_0(t)) =  \mathcal{L} (u_{-t}(0))  \rightarrow \mu, \quad \text{ as } t\rightarrow\infty. 
\end{align*}
Hence $\mu$ is an invariant measure. The uniqueness of invariant measures follows from Theorem \ref{260206.2029}. Setting $t=0$ and $\gamma=s$ in (\ref{260209.1618}) and then letting $\delta\rightarrow +\infty$, we obtain
\begin{equation*}
\begin{aligned}
  W_2(\mu_s,\mu) \leq \left(\mathbb{E} [\Vert u_{-s}(0) - \xi \Vert_{L^2}^2 ]\right)^{\frac{1}{2}}
  \leq  2 C_{\alpha,\kappa} e^{-\kappa s} \left(\Vert f\Vert_{L^2} + C \right) .
\end{aligned}
\end{equation*}
Therefore, the solution in $L^2$ is exponentially mixing with respect to the Wasserstein metric $W_2$.
\end{proof}

\vskip 0.6cm

We next provide more concrete conditions on the coefficients $b$ and $\sigma$ that guarantee the existence and uniqueness of invariant measures and exponential mixing. 
Introduce

\vskip 0.4cm

\noindent {\bf (H3)} $b$ is locally Lipschitz and satisfies
  \begin{align}\label{260607.1313}
    |b(x) - b(y)| \leq L_b |x-y|(1+|x|^{\nu-1}+|y|^{\nu-1}),\quad x,y\in\mathbb{R},
  \end{align}
for some constant $L_b\geq 0$, where $\nu\geq 1$.

Under conditions (H1), (H2) and (H3), we shall prove the existence of a unique invariant measure in $C_0^+([0,1])$ and exponential mixing with respect to the Wasserstein metric $W_p$. As a preparation, we will give a  uniform-in-time moment estimate for the solution to (\ref{1.1}), which is also of independent interest. For this estimate, we impose the following conditions.
\begin{itemize}
  \item [(S1)] The initial value $u_0 \in C_0^+([0,1])$. 
  \item [(S2)] The coefficient $\sigma$ is Lipschitz and satisfies 
\begin{align}\label{260616.1301}
  |\sigma(u)| \leq C_{\sigma} + G_{\sigma}|u|, \quad u\in\mathbb{R}.
\end{align}
  \item [(S3)] There exist constants $C_b\geq 0$ and $\theta\in\mathbb{R}$ such that
\begin{align}\label{260603.2134}
  ub(u)\leq C_b - \theta|u|^2, \quad \forall\  u\in\mathbb{R}.
\end{align}
\end{itemize}

\begin{proposition}\label{260604.1227}
Assume that (S1)--(S3) and (H3) hold, and that
\[
\theta> \left\{\frac{\sqrt{4p}}{\pi}\left(\frac{1}{2\sqrt{\pi}}\right)^{\frac{1}{p}+\frac{1}{2}} \left(\Gamma\Big(\frac{p-2}{6p}\Big)\right)^{\frac{3}{2}} \times 2G_{\sigma}\right\}^{\frac{4p}{p-2}}
\]
for some $p>2$. 
Then the unique solution $u$ to equation \eqref{1.1} satisfies
\begin{align*}
  \sup_{t\geq 0}\mathbb{E}\Big[\sup_{x\in [0,1]} |u(t,x)|^p\Big] <\infty.
\end{align*}
\end{proposition}

%\begin{remark}
%
%Theorem 3 of \cite{AM03} also establishes a related uniform-in-time estimate in a weighted $L^p$ space. However, the uniform-in-time estimate is obtained in a different way here. As a result, our condition on $\theta$ is expressed in terms of the linear-growth constant $G_{\sigma}$ in (S2), rather than the Lipschitz constant of $\sigma$, as in \cite{AM03}. In the present setting, the Lipschitz continuity of $\sigma$ is used mainly to ensure well-posedness of the equation, whereas the a priori estimate itself only uses the growth bound (\ref{260616.1301}). In this sense, the condition imposed on $\theta$ here is less restrictive than that in \cite{AM03}.
%
%\end{remark}

The proof of Proposition \ref{260604.1227} is inspired by that of Theorem 3 in \cite{AM03} and will be given at the end of this section.

\begin{theorem}\label{260604.1328}

Suppose that (H1), (H2) and (H3) hold. 
If 
\begin{align}\label{260621.2019}
  \alpha> \left\{\frac{\sqrt{4p}}{\pi}\left(\frac{1}{2\sqrt{\pi}}\right)^{\frac{1}{p}+\frac{1}{2}} \left(\Gamma\Big(\frac{p-2}{6p}\Big)\right)^{\frac{3}{2}} \times 2 L_{\sigma}\right\}^{\frac{4p}{p-2}}
\end{align}
for some $p>2$, then for every initial value $f \in C_0^+([0,1])$, the unique solution $u$ to equation \eqref{1.1} with initial value $f$ satisfies
\begin{align}\label{260621.2010}
  \sup_{t\geq 0}\mathbb{E}\Big[\sup_{x\in [0,1]} |u(t,x)|^p\Big] <\infty.
\end{align}
In particular, there exists a unique invariant measure $\mu$ in $C_0^+([0,1])$. Moreover, the solution is exponentially mixing with respect to the Wasserstein metric $W_p$.
\end{theorem}

\begin{remark}
If the initial value of (\ref{1.1}) belongs to $C_0^+([0,1])$, then (H1), (H2) and (H3) imply that equation (\ref{1.1}) admits a unique solution. 
\end{remark}

\begin{proof}
By Theorem \ref{260209.1719} (iii), 
we only need to prove \eqref{260621.2010} for any initial value $f\in C_0^+([0,1])$. 
Since $\alpha$ satisfies \eqref{260621.2019}, there exists a sufficiently small $\varepsilon>0$ such that
\[
\alpha-\varepsilon> \left\{\frac{\sqrt{4p}}{\pi}\left(\frac{1}{2\sqrt{\pi}}\right)^{\frac{1}{p}+\frac{1}{2}} \left(\Gamma\Big(\frac{p-2}{6p}\Big)\right)^{\frac{3}{2}} \times 2 L_{\sigma}\right\}^{\frac{4p}{p-2}}.
\]
Fix such $\varepsilon$. Taking $y=0$ in (H2) gives
\begin{align*}
  (b(u) - b(0))u \leq -\alpha |u|^2, \quad\forall\  u\in\mathbb{R},
\end{align*}
which implies
\begin{align}\label{260605.2001}
  ub(u) \leq -(\alpha - \varepsilon) |u|^2 + \frac{|b(0)|^2}{4\varepsilon}, \quad\forall\  u\in\mathbb{R}.
\end{align}
By (H1), 
\[
|\sigma(u) | \leq |\sigma(0)| + L_{\sigma} |u|. 
\]
Set $\theta:= \alpha -\varepsilon$ and $G_{\sigma}:= L_{\sigma} $. Then applying Proposition \ref{260604.1227} gives \eqref{260621.2010}. 
\end{proof}

\vskip 0.6cm

Next, we turn to the proof of Proposition \ref{260604.1227}. 
We begin with three auxiliary lemmas. Consider the auxiliary equation: 
\begin{align}\label{260603.1828}
	\left\{
	\begin{aligned}
		d v(t,x) &=[\frac{1}{2}\partial_{xx} v(t,x) - mv(t,x)]dt + \tilde{b}(v(t,x))dt + \sigma(v(t,x)) W(dt,dx)+L(dt,dx),\\
        v(t,x) & \geq 0, \quad t\geq 0, \ x\in [0,1],\\
        v(0,x)&=u_0(x), \quad x\in [0,1], \\
        v(t,0) &= v(t,1) = 0, \quad \forall\  t\geq 0,
	\end{aligned}
	\right.
\end{align}
where $m>0$.

\begin{lemma}\label{260603.2255}
Assume that (S1) and (S2) hold. Let $\tilde{v}$ be the unique solution to \eqref{260603.1828} with $\tilde{b}(v)  = - \gamma v$. 
Then, for any $m>0$, $p> 2$ and 
\begin{align}\label{260621.1621}
  \gamma > \left\{\frac{\sqrt{4p}}{\pi}\left(\frac{1}{2\sqrt{\pi}}\right)^{\frac{1}{p}+\frac{1}{2}} \left(\Gamma\Big(\frac{p-2}{6p}\Big)\right)^{\frac{3}{2}} \times 2G_{\sigma}\right\}^{\frac{4p}{p-2}},
\end{align}
we have
\begin{align}\label{260603.1925}
  \sup_{t\geq 0}\mathbb{E} \Big[ \sup_{x\in [0,1]}  |\tilde{v}(t,x)|^p \Big] <\infty.
\end{align}
\end{lemma}

\begin{proof}
The semigroup generated by $\frac{1}{2}\partial_{xx} - mI - \gamma I$ is $e^{-(m+\gamma) t}P_t$, where $P_t$ is defined in (\ref{260603.1931}). Hence, $\tilde{v}$ admits the following mild representation:
\begin{align*}
  \tilde{v}(t,x) = & e^{-(m+\gamma)t}P_t u_0(x) + \int_{0}^{t}\int_{0}^{1} e^{-(m+\gamma)(t-s)}p_{t-s}(x,y)\sigma(\tilde{v}(s,y)) W(ds,dy)  \nonumber\\
  & + \int_0^t \int_{0}^{1} e^{-(m+\gamma)(t-s)} p_{t-s}(x,y)L(ds,dy).
\end{align*}
Let $\widetilde{\Psi}$ be defined by
\begin{align*}
  \widetilde{\Psi}(t,x) = & e^{-(m+\gamma)t}P_t u_0(x) + \int_{0}^{t}\int_{0}^{1} e^{-(m+\gamma)(t-s)}p_{t-s}(x,y)\sigma(\tilde{v}(s,y)) W(ds,dy).
\end{align*}
By the proof of Theorem 1.4 of \cite{NP92} (see also Proposition 2.1 of \cite{XZ09}),
\begin{align}
\sup_{x\in [0,1]} |\tilde{v}(t,x)| \leq 2 \sup_{x\in [0,1]} |\widetilde{\Psi}(t,x)|, \quad \forall \ t\geq 0.
\end{align}
Set
\[
\mathcal{N}_{p,T}(\tilde{v}):= \sup_{0\leq t\leq T} \Big\Vert \sup_{x\in [0,1]}|\tilde{v}(t,x)| \Big\Vert_{L^p(\Omega)}.
\]
Then
\begin{align}\label{260603.2049}
  \mathcal{N}_{p,T}(\tilde{v}) \leq & 2\sup_{0\leq t\leq T}\sup_{x\in [0,1]}\big[ e^{-(m+\gamma)t}|P_t u_0(x)| \big] \nonumber\\
  & + 2\sup_{0\leq t\leq T}\left\Vert \sup_{x\in [0,1]} \Big|\int_{0}^{t}\int_{0}^{1} e^{-(m+\gamma)(t-s)}p_{t-s}(x,y)\sigma(\tilde{v}(s,y)) W(ds,dy)\Big| \right\Vert_{L^p(\Omega)} \nonumber\\
  =: & I(T) + II(T).
\end{align}
By \eqref{260208.1019}, we have
\begin{align}\label{260603.2156}
  I(\infty)\leq 2\sup_{x\in [0,1]} |u_0(x)|.
\end{align}  
By \eqref{260617.2015} and \eqref{260616.1301}, the term $II$ can be estimated as follows
\begin{align}\label{260604.1025}
  II(T) 
%  = & \sup_{0\leq t\leq T}\sup_{x\in [0,1]} \left\{\mathbb{E}\left| \int_{0}^{t}\int_{0}^{1} e^{-(m+\gamma)(t-s)}p_{t-s}(x,y)\sigma(\tilde{v}(s,y)) W(ds,dy) \right|^p\right\}^{\frac{1}{p}}  \nonumber\\
  \leq & 2C_{p,m+\gamma} \sup_{0\leq s\leq T}\sup_{y\in [0,1]} \Vert \sigma(\tilde{v}(s,y))\Vert_{L^p(\Omega)}  \nonumber\\
  \leq & 2C_{p,m+\gamma}  C_{\sigma} + 2C_{p,m+\gamma} G_{\sigma}  \sup_{0\leq s\leq T}\sup_{y\in [0,1]} \Vert \tilde{v}(s,y) \Vert_{L^p(\Omega)} \nonumber\\
  \leq & 2C_{p,m+\gamma}  C_{\sigma} + 2C_{p,m+\gamma} G_{\sigma}  \sup_{0\leq s\leq T} \Big\Vert \sup_{y\in [0,1]}|\tilde{v}(s,y)| \Big\Vert_{L^p(\Omega)},
\end{align}
where $C_{p,m+\gamma}$ is the constant appearing in \eqref{260617.2015} with $\beta$ replaced by $m+\gamma$. 
Note that \eqref{260621.1621} implies
\[
2C_{p,m+\gamma} G_{\sigma} <1.
\]
Combining (\ref{260603.2049}), (\ref{260603.2156}) and (\ref{260604.1025}) yields
\begin{align}\label{260605.1228}
  \mathcal{N}_{p,T}(\tilde{v}) \leq 2\sup_{x\in [0,1]}|u_0(x)| + 2C_{p,m+\gamma} C_{\sigma}
    + 2C_{p,m+\gamma} G_{\sigma} \times\mathcal{N}_{p,T}(\tilde{v}).
\end{align}
By Theorem 2.1 of \cite{XZ09}, $\mathcal{N}_{p,T}(\tilde{v})<\infty$ for any $T>0$. 
Therefore, (\ref{260605.1228}) gives
\[
\sup_{T\geq 0}\mathcal{N}_{p,T}(\tilde{v})<\infty,
\]
which proves (\ref{260603.1925}).
\end{proof}

Recall condition (S3), and set 
\begin{align}\label{260607.1423}
  C_{b,\theta}:= C_b + \theta + \sup_{|x|\leq 1}|b(x)|. 
\end{align}
It follows from (\ref{260603.2134}) that
\begin{align}\label{260604.1239}
  \begin{cases}
    b(u)\leq -\theta u + C_{b,\theta}, & \forall\ u\geq 0, \\
    b(u)\geq -\theta u - C_{b,\theta}, & \forall\ u\leq 0.
  \end{cases}
\end{align}

\begin{lemma}\label{260604.1225}
 Assume that (S1)--(S3) and (H3) hold. Let $p>2$. Define
  \begin{align*}
    h_{+}(u):=
  \begin{cases}
    C_{b,\theta}, & u\geq 0,\\
    -\gamma u + \sup_{u\leq v\leq 0}\big(b(v) - b(0) + C_{b,\theta}\big), & u\leq 0,
  \end{cases}
  \end{align*}
where $\gamma$ satisfies \eqref{260621.1621} with $p$ replaced by $p\nu$. Let $v_{+}$ be the unique solution to \eqref{260603.1828} with $\tilde{b}=h_{+}$. If
\begin{align}\label{260621.1647}
  m > \left\{\frac{\sqrt{4p}}{\pi}\left(\frac{1}{2\sqrt{\pi}}\right)^{\frac{1}{p}+\frac{1}{2}} \left(\Gamma\Big(\frac{p-2}{6p}\Big)\right)^{\frac{3}{2}} \times 2G_{\sigma}\right\}^{\frac{4p}{p-2}},
\end{align}
then
\begin{align}\label{260603.2314}
  \sup_{t\geq 0}\mathbb{E}\Big[\sup_{x\in [0,1]} |v_{+}(t,x)|^p\Big] <\infty.
\end{align}

\end{lemma}

\begin{proof}
By (H3), $h_+$ also satisfies the polynomial local Lipschitz condition (H3) as $b$. Moreover,
\[
uh_{+}(u) \leq \frac{1}{2}(|u|^2 + C_{b,\theta}^2).
\] 
Therefore, the argument in the proof of Theorem 1 of \cite{AM03} (see also Theorem 3.4.1 of \cite{MZ99}), together with Theorem 2.1 of \cite{XZ09}, applies to (\ref{260603.1828}) with $\tilde{b}=h_{+}$ and yields the existence and uniqueness of the solution $v_{+}$. It remains to prove the uniform estimate (\ref{260603.2314}). 

Consider an auxiliary equation
\begin{align}\label{260605.1008}
	\left\{
	\begin{aligned}
		d v(t,x) &=[\frac{1}{2}\partial_{xx} v(t,x) - mv(t,x)]dt + h_+(v_+(t,x))dt + \sigma(v(t,x)) W(dt,dx) + L(dt,dx),\\
        v(t,x) & \geq 0, \quad t\geq 0, \ x\in [0,1],\\
        v(0,x)&=u_0(x), \quad x\in [0,1],\\
        v(t,0) &= v(t,1) = 0, \quad \forall\  t\geq 0.
	\end{aligned}
	\right.
\end{align}
Since $\sigma$ is Lipschitz and the drift term $h_+(v_+)$ does not depend on $v$, the above equation has a unique solution. Meanwhile, $v_+$ also satisfies the above equation, which implies that the unique solution to (\ref{260605.1008}) is exactly $v_+$. 
Now construct a sequence of approximating solutions $v_+^n$. Set $v_+^0(t) = u_0$ for any $t\geq 0$, and then iteratively define
\begin{align*}
  v_+^{n+1}(t,x) = & e^{-mt}P_t u_0(x) + \int_{0}^{t}\int_{0}^{1} e^{-m(t-s)}p_{t-s}(x,y)h_+(v_+(s,y))dsdy \nonumber\\
& + \int_{0}^{t}\int_{0}^{1} e^{-m(t-s)}p_{t-s}(x,y)\sigma(v_+^n(s,y))W(ds,dy) \nonumber\\
&  + \int_{0}^{t}\int_{0}^{1} e^{-m(t-s)}p_{t-s}(x,y)L(ds,dy).
\end{align*}
Then
\begin{align}\label{260605.1117}
  \lim_{n\rightarrow\infty} v_+^{n}(t,x) = v_+(t,x) \quad\text{ in } L^2(\Omega),\quad \forall\ (t,x)\in [0,\infty)\times[0,1].
\end{align}
See, e.g., Chapter 3 of Walsh \cite{WA}.

Let $\Psi_+$ be defined by
\begin{align*}
  \Psi_+(t,x) = & e^{-m t}P_t u_0(x) + \int_{0}^{t}\int_{0}^{1} e^{-m(t-s)}p_{t-s}(x,y)h_+(v_+(s,y))dsdy \nonumber\\  
  & + \int_{0}^{t}\int_{0}^{1} e^{-m(t-s)}p_{t-s}(x,y)\sigma(v_+^n(s,y))W(ds,dy).
\end{align*}
By the proof of Theorem 1.4 of \cite{NP92} (see also Proposition 2.1 of \cite{XZ09}),
\begin{align}
\sup_{x\in [0,1]} |v_+^{n+1}(t,x)| \leq 2 \sup_{x\in [0,1]} |\Psi_+(t,x)|, \quad \forall \ t\geq 0.
\end{align}
Set
\[
\mathcal{N}_p(v_+^n):= \sup_{t\geq 0}\Big\Vert \sup_{x\in [0,1]}|v_+^n(t,x)|\Big\Vert_{L^p(\Omega)}.
\]
Then
\begin{align}\label{260604.1046}
  \mathcal{N}_p (v_+^{n+1}) \leq & 2\sup_{t\geq 0}\sup_{x\in [0,1]}\big[ e^{-mt}|P_t u_0(x)| \big] \nonumber\\
  & + 2\sup_{t\geq 0}\left\Vert \sup_{x\in [0,1]}\left| \int_{0}^{t}\int_{0}^{1} e^{-m(t-s)}p_{t-s}(x,y) h_+(v_+(s,y)) dsdy \right| \right\Vert_{L^p(\Omega)} \nonumber\\
  & + 2\sup_{t\geq 0}\left\Vert \sup_{x\in [0,1]}\left|\int_{0}^{t}\int_{0}^{1} e^{-m(t-s)}p_{t-s}(x,y)\sigma(v_{+}^n(s,y)) W(ds,dy) \right| \right\Vert_{L^p(\Omega)} \nonumber\\
  =: & I_1 + I_2 + I_3.
\end{align}

We first prove that $I_2<\infty$ by using the comparison principle. By (\ref{260207.1152}), 
\begin{align}\label{260603.2248}
  I_2\leq & 2\sup_{t\geq 0} \left\Vert \int_{0}^{t}  e^{-m(t-s)} \sup_{y\in [0,1]}| h_+(v_+(s,y))| ds  \right\Vert_{L^p(\Omega)} \nonumber\\  
  \leq & \frac{2}{m} \sup_{s\geq 0}\Big\Vert  \sup_{y\in [0,1]}| h_+(v_+(s,y))| \Big\Vert_{L^p(\Omega)}.
\end{align}
By the definition of $h_+$ and (H3), 
\begin{align}
\label{260603.2241} |h_+(u)|\leq & C_{b,\theta} + (\gamma + L_b)|u| + L_b |u|^{\nu}, \quad \forall\ u\in\mathbb{R},\\
\label{260603.2242}  h_+(u) \geq & -\gamma u,\quad \forall\ u\in\mathbb{R}.
\end{align}
It follows from (\ref{260603.2242}) and the comparison principle that
\[
v_+(t,x) \geq \tilde{v}(t,x), \quad \forall\ (t,x)\in [0,\infty)\times [0,1].
\]
Since $h_+$ is nonincreasing and takes values in $[C_{b,\theta},\infty)$, (\ref{260603.2241}) gives
\begin{align}\label{260603.2247}
  |h_+(v_+(s,y))| \leq |h_+(\tilde{v}(s,y))| \leq C_{b,\theta} + (\gamma + L_b)|\tilde{v}(s,y)| + L_b |\tilde{v}(s,y)|^{\nu}.
\end{align}
Substituting (\ref{260603.2247}) into (\ref{260603.2248}) yields
\begin{align*}
  I_2\leq & \frac{2}{m} \sup_{s\geq 0}\Big\Vert  C_{b,\theta} + (\gamma + L_b)\sup_{y\in [0,1]}| \tilde{v}(s,y)| + L_b \sup_{y\in [0,1]}|\tilde{v}(s,y)|^{\nu} \Big\Vert_{L^p(\Omega)} \nonumber\\  
  \leq & \frac{2C_{b,\theta}}{m} + \frac{2(\gamma + L_b)}{m} \sup_{s\geq 0}\Big\Vert\sup_{y\in [0,1]}| \tilde{v}(s,y)| \Big\Vert_{L^p(\Omega)} + \frac{2L_b}{m} \sup_{s\geq 0}\Big\Vert\sup_{y\in [0,1]}| \tilde{v}(s,y)| \Big\Vert_{L^{p\nu}(\Omega)}^{\nu} \nonumber\\
  < & \infty,
\end{align*}
where we have used Lemma \ref{260603.2255}.

The terms $I_1$ and $I_3$ can be estimated as in (\ref{260603.2156}) and (\ref{260604.1025}). This gives
\begin{gather}
\label{260603.2311}  I_1 \leq 2\sup_{x\in [0,1]}|u_0(x)|, \\
\label{260603.2306}  I_3  \leq 2 C_{p,m} C_{\sigma}  + 2 C_{p,m} G_{\sigma} \mathcal{N}_p(v_+^n),
\end{gather}
where $C_{p,m}$ is the constant appearing in \eqref{260617.2015} with $\beta$ replaced by $m$.
%\[
%C_{p,m} := \frac{1}{2\sqrt{2m}}
%\]
Combining (\ref{260604.1046}), (\ref{260603.2311}) and (\ref{260603.2306}) yields
\begin{align}\label{260605.1114}
  \mathcal{N}_p(v_+^{n+1}) \leq 2\sup_{x\in [0,1]}|u_0(x)| + I_2 + 2 C_{p,m} C_{\sigma} + 2C_{p,m} G_{\sigma} \mathcal{N}_p(v_+^n).
\end{align}
Since $u_0$ is bounded, we have $\mathcal{N}_p(v_+^{0})<\infty$ and hence $\mathcal{N}_p(v_+^{1})<\infty$. Note that
\eqref{260621.1647} is equivalent to $2C_{p,m} G_{\sigma}<1$. Iterating (\ref{260605.1114}) yields
\[
\sup_{n\geq 1}\mathcal{N}_p(v_+^n)<\infty.
\]
By Fatou's lemma and (\ref{260605.1117}),
\[
\mathcal{N}_p(v_+) \leq \liminf_{n\rightarrow\infty} \mathcal{N}_p(v_+^n)<\infty,
\]
which proves (\ref{260603.2314}).
\end{proof}

\begin{lemma}\label{260604.1327}
Assume that (S1)--(S3) and (H3) hold. Let $p>2$. Define
  \begin{align*}
    h_{-}(u):=
  \begin{cases}
        -\gamma u + \inf_{0\leq v\leq u}\big(b(v) - b(0) - C_{b,\theta}\big), & u\geq 0, \\
        - C_{b,\theta}, & u\leq 0,\\
  \end{cases}
  \end{align*}
where $\gamma$ satisfies \eqref{260621.1621} with $p$ replaced by $p\nu$. Let $v_{-}$ be the unique solution to (\ref{260603.1828}) with $\tilde{b}=h_{-}$. If 
$m$ satisfies \eqref{260621.1647}, 
then
\begin{align*}
  \sup_{t\geq 0}\mathbb{E}\Big[ \sup_{x\in [0,1]}|v_{-}(t,x)|^p \Big] <\infty.
\end{align*}

\end{lemma}

\begin{proof}
  The proof is the same as that of Lemma \ref{260604.1225}, except for the estimate of $I_2$. We give the necessary modifications.
  
By the definition of $h_{-}$ and condition (H3), 
\begin{align}
\label{260604.1216} |h_{-}(u)|\leq & C_{b,\theta} + (\gamma + L_b)|u| + L_b |u|^{\nu}, \quad \forall\ u\in\mathbb{R},\\
\label{260604.1215}  h_{-}(u) \leq & -\gamma u,\quad \forall\ u\in\mathbb{R}.
\end{align}
It follows from (\ref{260604.1215}) and the comparison principle that
\[
v_{-}(t,x) \leq \tilde{v}(t,x), \quad \forall\ (t,x)\in [0,\infty)\times[0,1].
\]
Since $h_-$ is nonincreasing and takes values in $(-\infty, -C_{b,\theta}]$, by (\ref{260604.1216}) we still get
\[
  |h_{-}(v_{-}(s,y))| \leq |h_{-}(\tilde{v}(s,y))| \leq C_{b,\theta} + (\gamma + L_b)|\tilde{v}(s,y)| + L_b |\tilde{v}(s,y)|^{\nu}.
\]
\end{proof}

\begin{proof}[Proof of Proposition \ref{260604.1227}]
By the argument in the proof of Theorem 1 of \cite{AM03} (see also Theorem 3.4.1 of \cite{MZ99}) and Theorem 2.1 of \cite{XZ09}, there exists a unique solution $u$ to equation (\ref{1.1}) under conditions (S1)--(S3) and (H3).  

Let $b_{\theta}(u) = b(u)+\theta u$. 
Then $u$ also solves (\ref{260603.1828}) with $m$ satisfying \eqref{260621.1647} and $\tilde{b}(u)=b_{\theta}(u)$. By (\ref{260604.1239}) and the definitions of $h_+$ and $h_{-}$, a direct computation gives
\begin{align*}
  h_{-}(u) \leq b_{\theta}(u)\leq h_+(u), \quad \forall \ u\in\mathbb{R},
\end{align*}
for any $\gamma\geq 0$. 
The comparison principle then yields
\[
v_{-}(t,x) \leq u(t,x) \leq v_{+}(t,x), \quad \forall \ (t,x)\in [0,\infty)\times[0,1].
\]
Therefore, by Lemma \ref{260604.1225} and \ref{260604.1327}, 
\[
\sup_{t\geq 0}\mathbb{E} \Big[ \sup_{x\in [0,1]} |u(t,x)|^p \Big] \leq \sup_{t\geq 0} \mathbb{E} \Big[ \sup_{x\in [0,1]}|v_+(t,x)|^p\Big] + \sup_{t\geq 0} \mathbb{E} \Big[ \sup_{x\in [0,1]}|v_{-}(t,x)|^p \Big] <\infty,
\]
which completes the proof of Proposition \ref{260604.1227}.
\end{proof}

\vskip 0.3cm
\noindent{\large\bf Acknowledgements}

This work is partially supported by the National Key R\&D Program of China (No. 2022YFA1006001), the National Natural Science Foundation of China (Nos. 12131019, 12571158, 12371151), and the Fundamental Research Funds for the Central Universities (No. WK0010000081).

\bibliographystyle{plain}

\end{document}